\documentclass[12pt]{article}

\usepackage[UKenglish]{babel}
\usepackage[margin=2.25cm]{geometry}
\usepackage{amsmath,amsthm,amssymb,latexsym,mathrsfs,amsfonts}
\usepackage{bm,color,graphicx}

\def\leaderfill{\leaders\hbox to .25em{\hss.\hss}\hfill}

\newcommand{\red}[1]{\textcolor{black}{#1}}
\newcommand{\blue}[1]{\textcolor{blue}{#1}}

\newtheorem{theorem}{\bf Theorem}[section]
\newtheorem{lemma}{\bf Lemma}[section]

\newtheorem{corollary}{\bf Corollary}[section]
\newtheorem{remark}[theorem]{Remark}

\def\rd{\text{d}}
\def\re{\text{e}}
\def\ri{\text{i}}
\def\eps{{\varepsilon}}

\def\qand{\quad\mbox{and}\quad}
\def\fr{\mbox{$\frac{1}{2}$}}

\def\R{\mathbb{R}}

\def\N{\mathbb{N}}

\def\Dt{\partial_T}
\def\Chi{G}
\def\({\left(}
\def\){\right)}

\def\bk{{\bf k}}
\def\bv{{\bf v}}

\def\bw{{\bm\omega}}
\def\bq{{\bf q}}
\def\bO{{\bm\Omega}}

\def\bu{{\bf u}}
\def\bv{{\bf v}}

\def\bw{{\bm\omega}}

\def\eb{{\nu}}

\begin{document}

\begin{center} 
\textbf{\Large A proof of validity for multiphase}\\[2mm] 

\textbf{\Large Whitham modulation theory\footnote{\blue{Revised version, 22 September 2020}}}

\vspace{.75cm} 

\textsf{\large Thomas J. Bridges$^1$, Anna Kostianko$^{1,2}$ {\small and} Guido Schneider$^3$} 
\vspace{.25cm} 

\textit{1. Department of Mathematics, University of Surrey, Guildford GU2 7XH, UK} 

\textit{2. School of Mathematics and Statistics, Lanzhou University, Lanzhou 730000 P.R. China} 

\textit{3. Institut f\"ur Analysis, Dynamik und Modellierung, 
Universit\"at Stuttgart, Pfaffenwaldring 57,
70569 Stuttgart, Germany} 

\end{center}

\begin{abstract}
\noindent It is proved
that approximations which are obtained as solutions of the multiphase Whitham
modulation equations stay close
to solutions of the original equation on a natural time scale. The class
of nonlinear wave equations chosen for the starting point is
coupled nonlinear Schr\"odinger equations.  These equations are
not in general integrable, but they
have an explicit family of multiphase wavetrains that generate
multiphase Whitham equations which may be elliptic,
hyperbolic, or of mixed type. Due to the change of type, the function space
setup is based on Gevrey spaces with initial data analytic in a strip in the complex
plane.  In these spaces a Cauchy-Kowalevskaya-like existence and uniqueness theorem is proved.  Building on this theorem and higher-order approximations
to Whitham theory, a rigorous comparison of solutions, of the
coupled nonlinear Schr\"odinger equations and
the multiphase Whitham modulation equations, is obtained.
\end{abstract}

\section{Introduction}
\setcounter{equation}{0}
\label{sec-intro}

Given a periodic travelling wave of a conservative nonlinear wave equation, generated by
a Lagrangian, Whitham modulation theory, in its simplest one-phase form,
is a perturbation theory
where the wavenumber $k$ and frequency $\omega$ of the travelling wave are
perturbed and allowed to vary slowly in time and space,
thereby capturing modulation of the basic wave.

The theory reduces the original
nonlinear wave equation to a pair of first order quasilinear PDEs
\begin{equation}\label{wmes-1}
\partial_T q= \partial_X \Omega \qand \partial_T \mathscr{A}(\omega+\Omega,k+q)+ \partial_X \mathscr{B}(\omega+\Omega,k+q) = 0 \,,
\end{equation}
where $\Omega(X,T,\eps)$
is the slowly varying frequency and $q(X,T,\eps)$ is the slowly varying wavenumber.
The independent variables in (\ref{wmes-1}) are slow time and space coordinates,
$T=\eps t$ and $X=\eps x$, with $0< \eps\ll1$.
The function $\mathscr{A}$ is the wave action and
$\mathscr{B}$ is the wave action flux, and they are
determined from a given Lagrangian, and satisfy
$\partial_k \mathscr{A}=\partial_\omega \mathscr{B}$.

Analysis of the Whitham modulation equations (WMEs) in (\ref{wmes-1})
then leads to deductions
about the effect of perturbations on the original periodic travelling wave. For
example when the pair (\ref{wmes-1}) is elliptic (hyperbolic) the original
periodic travelling wave is unstable (stable) to long wave perturbations,
with appropriate hypotheses
\red{(cf.\ \textsc{Benzoni-Gavage, et al.}~\cite{bgnr})}.
There is
now a vast literature on the reduction process, asymptotics,
and analysis of the WMEs for a wide range of nonlinear wave equations,
generated by a Lagrangian, in the case
where the basic state is a single-phase travelling wave (e.g.\
\textsc{Whitham}~\cite{w74}, \textsc{Kamchatnov}~\cite{k00},
\textsc{Biondini et al.}~\cite{behm16}, \textsc{Bridges}~\cite{tjb17}
and references therein).

An obvious question is how accurate the solutions of
(\ref{wmes-1}) are when compared to solutions of the original equation.
\red{From a physical point of view, the interest in validity is that an
  approximation should represent the solutions of the original equation
  as closely as possible.  There are many examples where approximate
  equations are \emph{not valid} (a history with examples is in
  \textsc{Sunny}~\cite{sunny}), and so there is no reason to assume
  {\it a priori} that even asymptotically valid approximations are rigorously
  valid.} A rigorous comparison between approximate and exact solutions
requires introduction of a metric and an existence theory
in a function space large enough to accommodate the range of solutions expected
of nonlinear wave equations on the real line.

A proof of the validity of the WMEs (\ref{wmes-1})
has been given
by \textsc{D\"ull \& Schneider}~\cite{ds09} when the original equation is
the cubic nonlinear Schr\"odinger (NLS) equation
\begin{equation}\label{nls-1}
\ri \partial_t \Psi + \partial_{x}^2
\Psi+ \gamma |\Psi|^2\Psi = 0 \,,
\end{equation}
where $\Psi(x,t)$ is complex valued, $\gamma=\pm1$, $x\in\R$ and $t\geq 0$.
First, $\Psi$ is expressed in the form
\begin{equation}\label{Psi-r-phi}
\Psi(x,t) = {\rm exp}\big(r(x,t)+\ri\phi(x,t)\big)\,,
\end{equation}
where $r$ and $\phi$ are real-valued,
and then a pair of equations for $r$ and $v:=\partial_x \phi$ is derived. The exact
equations for $r$ and $v$ are recast in terms of the same independent variables,
$X=\eps x$ and $T=\eps t$, as in the reduced equations,
\begin{equation}\label{ru-check-def}
r(x,t) = \check{r}(X,T,\eps) \qand v(x,t) = \check{v}(X,T,\eps)\,.
\end{equation}
The strategy is then to compare the solutions $\check{r}(X,T,\eps)$ and
$\check{v}(X,T,\eps)$ of the exact equations to solutions of the
WMEs for $\varepsilon>0$ sufficiently small.

The WMEs for (\ref{nls-1})
are deduced by modulating the basic travelling wave solution,
\begin{equation}\label{nls-1-tw}
\Psi_0(x,t) = \re^{r_0+\ri(kx+\omega t+\theta_0)}\,,\quad \mbox{with}\quad
\omega + k^2 - \gamma \re^{2r_0}=0\,.
\end{equation}
The modulation mapping
\begin{equation}\label{modulation-mapping}
r_0\mapsto r_0 + r^*(X,T) \qand k \mapsto k + v^*(X,T)\,,
\end{equation}
leads to a form of the WMEs (\ref{wmes-1}) in terms of
$r^*(X,T)$ and $v^*(X,T)$.

The rigorous approximation result for (\ref{nls-1})
is as follows. Given a solution of the
NLS equation in coordinates
(\ref{ru-check-def}) and a solution of the
WMEs in coordinates (\ref{modulation-mapping}), with
initial data satisfying
\begin{equation}\label{ru-initialdata}
\left\|(\check{r}(X,T,\eps),\check{v}(X,T,\eps)\big)\big|_{T=0} -
\big(r^*(X,T),v^*(X,T)\big)\big|_{T=0}\right\| = \mathcal{O}(\varepsilon)\,,
\end{equation}
in a suitably chosen norm $ \| \cdot \| $,
the main validity result in \cite{ds09} is
\begin{equation}\label{sup-sup-identity}
\sup_{T\in[0,T_1]}\sup_{X\in\R}\big|
\big(\check{r}(X,T,\eps),\check{v}(X,T,\eps)\big)-
\big(r^*(X,T),v^*(X,T)\big)\big| \leq C_2\eps
\quad\mbox{for all}\ \eps\in(0,\eps_0)\,,
\end{equation}
where $\eps_0$, $C_2$ and $T_1$ are all positive constants. From this
estimate we can conclude that the reduction of the NLS equation
(\ref{nls-1}), in the neighbourhood of the family of periodic
travelling waves (\ref{nls-1-tw}), to the WMEs,
is valid on the natural time scale, $t=\mathcal{O}(\eps^{-1})$.

\red{For the single NLS equation, there
  is a close connection between reduction to the WMES and the semiclassical
  approximation of NLS, and this problem has been extensively studied,
  particularly in the case of defocussing NLS (e.g.\ \cite{jlm99,k00} and references therein).  A related problem is the dispersionless limit of the KdV equation \cite{k00,Ego,Lax,bhj}.  The
limit in both cases can be very delicate and oscillatory. The most detailed results on the limit have been
obtained using integrability of the NLS equation and KdV equations.
}

In moving from the single NLS to coupled NLS, and moving from single-phase
WMEs to multiphase WMEs, new challenges arise.  Integrability is rare
and so methods that do not require integrability are needed.
\red{An additional challenge is the fact
that the multi-phase WMEs do not, in general, preserve the
characteristic type (elliptic, hyperbolic, or mixed) under evolution.
In addition, even if the hyperbolicty is preserved on some specific
   trajectory of multi-phase WMEs, the linear stability of the background wave-train may be lost
   during the evolution. For these reasons, a validity result in Sobolev spaces for multi-phase WMEs
   is problematic no matter how regular these Sobolev spaces are.  On the
   other hand, by taking initial data that is analytic in a strip about
   the real axis, Gevrey spaces capture all
   trajectories regardless of whether the characteristics are elliptic,
   hyperbolic, or mixed.}

A proof of validity
requires three steps (e.g.\ \cite{ds09,su17}): (i) a local existence and uniqueness theory
for the WMEs (\ref{wmes-1}), (ii) a local existence and
uniqueness theory for solutions of the original equation (\ref{nls-1}),
and (iii) an approximation theory for the difference between the two
solutions. The backbone of all three steps is the choice of
function space, and the choice of coordinates, with the latter
chosen to facilitate the analysis.
The function space has to be large enough to include all bounded solutions
on the real line, and it has to account for the fact that the WMEs
may be elliptic and so not well posed. Inspired by \cite{ds09},
the function space to be used here is a scale of Gevrey spaces.

A two-phase wavetrain is
a solution of a nonlinear wave equation of the form
\[
{\bf u}(x,t) = \widehat{\bf u}(\theta_1,\theta_2)\,,\quad
\theta_j=k_jx+\omega_jt+\theta_j^0\,,\ j=1,2,
\]
with
\[
\widehat{\bf u}(\theta_1+2\pi,\theta_2)=\widehat{\bf u}(\theta_1,\theta_2)
\qand \widehat{\bf u}(\theta_1,\theta_2+ 2\pi)=\widehat{\bf u}(\theta_1,\theta_2)\,.
\]
The wavenumbers ${\bf k}=(k_1,k_2)$ and
frequencies ${\bm\omega}=(\omega_1,\omega_2)$ are in general distinct
and $\theta_1^0$ and $\theta_2^0$ are constant phase shifts.
A multiphase wavetrain is the generalization of this form
to $N$-phases with $N$ finite.

The modulation of multiphase wavetrains, from the perspective of Whitham theory
starting from a Lagrangian, was first studied by
\textsc{Ablowitz \& Benney}~\cite{ab70}.
They derived the conservation of wave action for scalar fields with two
phases in detail, and showed how the theory generalized to $N$ phases.
Examples in \textsc{Ablowitz}~\cite{a72} show that in
general one should expect small divisors, but weakly nonlinear solutions
could still be obtained. However for integrable systems, multiphase
averaging and the WMEs are robust and rigorous, without small
divisors, and a general
theory can be obtained. There is now a vast literature on
multiphase WMEs for integrable systems (see \textsc{Flaschka et al.}~\cite{ffm80} and its citation trail).
On the other hand if the system is not integrable,
but there is an $N$-fold
toral symmetry,
then again a theory for conservation of wave action
and multiphase WMEs can
be developed without small divisors and smoothly varying $N$-phase
wavetrains (see \textsc{Ratliff}~\cite{r17}).
 The action of an $N-$fold toral
   symmetry, $\mathbb{T}^N$,
   in the context of coupled NLS equations, can be represented as
  \[
  (\psi_1,\ldots,\psi_N) \mapsto ( e^{i \theta_1}\psi_1,\ldots e^{i \theta_N}\psi_N)
  \,,\quad (\theta_1,\ldots,\theta_N)\in S^1\times\cdots\times S^1:=\mathbb{T}^N\,,
  \]
  with the $N=2$ case needed for the coupled NLS equation introduced
  below.
 In essence \red{the basic wavetrain is aligned with the toral
 symmetry}, and the conservation of wave action is replaced by the conservation
law generated by the symmetry. It is this latter class of multiphase
WMEs whose validity is of interest here.

In approaching the validity problem for multiphase WMEs,
a general theory starting from an abstract Lagrangian is at present intractable.
Therefore, we restrict attention to
the case of modulation of two-phase wavetrains of
coupled nonlinear Schr\"odinger (CNLS) equations
where a rigorous and complete reduction theory can be obtained.

CNLS equations arise in a wide range of applications
(e.g.\  models for Bose-Einstein condensates \cite{sb09},
the theory of water waves and rogue waves \cite{roskes,dls19}, nonlinear optics
\cite{kfc15}). For definiteness, we take the following form for the
CNLS equations as a starting point
\begin{equation}\label{cnls}
\begin{array}{rcl}
\displaystyle \ri {\partial_t \Psi_1}+
{\partial^2_x \Psi_1} +
\gamma_1|\Psi_1|^2\Psi_1 + \alpha |\Psi_2|^2\Psi_1 &=& 0,  \\[2mm]
\displaystyle \ri{\partial_t \Psi_2}+
{\partial_x^2 \Psi_2} + \alpha
|\Psi_1|^2\Psi_2 + \gamma_2|\Psi_2|^2\Psi_2 &=& 0\,.
\end{array}
\end{equation}
These equations are known to be integrable for only very special
values of the coefficients \cite{dls18}.  Here integrability is not
assumed and the coefficients are free to take any values with
$\alpha\in\R\setminus\{0\}$ and $\gamma_j \in  \{-1, 1\}$ and the
non-degeneracy constraint
\begin{equation}\label{gamma-alpha-condition}
  \gamma_1\gamma_2 - \alpha^2 \neq 0\,.
\end{equation}
The CNLS equation (\ref{cnls})
has an explicit four-parameter family of two-phase wavetrains.
Modulation of
this family of two-phase wavetrains generates a system of multiphase WMEs.

To prove validity the pair of equations (\ref{cnls})
is first transformed to the form
\begin{equation}\label{u-M-form}
  {\bf u}_T = {\bf M}({\bf u})\bu_X + \eps^2 {\bf F}(D_X^3\bu)\,,
\end{equation}
  where ${\bf u}$ has four components (based on the real and imaginary parts
  of $\Psi_1$ and $\Psi_2$),
  ${\bf M}(\bu)$ is entire, ${\bf F}$ is a polynomial in
  $D_X^k\bu=(\bu,\partial_X\bu,\ldots,\partial_X^k\bu)$.  The independent
  variables $x,t$ are scaled as $T=\eps t$ and $X=\eps X$.

  The principal advantage of the form (\ref{u-M-form}) is that when $\eps=0$
  it reduces to the multiphase WMEs.  Hence if
  ${\bf u}^*$ is a solution of the WMEs
  \begin{equation}\label{M-wmes}
    \partial_T \bu^* = {\bf M}(\bu^*)\partial_X\bu^*\,,
  \end{equation}
  then the validity proof is obtained by studying the difference
  $\|\bu(X,T,\eps) - \bu^*(X,T)\|$
  as a function of time in an appropriate function space, which in
  this case is a Gevrey space.

After choosing appropriate coordinates
and introducing the properties of Gevrey spaces, the steps of the
validity theory are set in motion.
Firstly existence is proved for the WMEs in
\S\ref{sec-CK-gevrey}\ref{subsec-MPWMEs-existence}, using an abstract
Cauchy-Kowalevskaya-type theorem developed in \S\ref{sec-CK-gevrey}.
Secondly, improved approximations (higher-order Whitham theory) are
obtained in \S\ref{sec-approx}.  The generated perturbation series is
not necessarily convergent, and so the exact solution (perturbation series
plus remainder) is studied in \S\ref{sec-exactsolutions}.
In \S\ref{sec-exactsolutions} an abstract theory is developed for
general systems of the form (\ref{M-wmes}), and then a summary
for the special case of CNLS is given in \S\ref{sec-exactsolutions}\ref{sec-summary},
thereby completing the proof of validity.
In the concluding remarks section implications and generalizations
are discussed.  \red{In the appendix some supplementary results on
Gevrey spaces are proved.}
\medskip

\noindent{\bf Remarks.}
(a) For notational simplicity we have restricted to \eqref{cnls},
although it is not the most general form of the CNLS equations.
By rescaling $ x,t,\Psi_1$ and $ \Psi_2 $ in general a normal form of the CNLS equations with
one additional parameter in front of one of the $ x $-derivative terms is obtained. However, for our purposes the above form (\ref{cnls}) is not a restriction.
(b) Throughout the paper, many different constants are denoted with
the same symbol $ C $ if they can be
chosen independently of the small, $0 < \varepsilon \ll 1$, perturbation parameter.

\section{Formal derivation of multiphase WMEs}
\setcounter{equation}{0}
\label{sec-formalderivation-MP-WMEs}

The basic two-phase wavetrain of CNLS is
\begin{equation}\label{twophase-wavetrains}
\Psi_j(x,t) = \Psi_j({\bm\theta},\bw,\bk) := \psi_j(\bw,\bk)\re^{\ri\theta_j(x,t)}\,,\quad
\theta_j(x,t) = k_jx+\omega_jt+\theta_j^0\,,\quad j=1,2\,,
\end{equation}
with ${\bm\theta}=(\theta_1,\theta_2)$,
$\bw=(\omega_1,\omega_2)$ and $\bk=(k_1,k_2)$, and the
amplitudes ${\bm\psi}=(\psi_1,\psi_2)$ are real valued.
Substitution
into the governing equations (\ref{cnls}) generates
a relationship between the amplitudes and the frequencies and
wavenumbers,
\begin{equation}\label{basic-state}
\begin{array}{rcl}
\psi_1(\bw,\bk)^2 &=&\displaystyle
\frac{1}{\beta}\bigg(\gamma_2(\omega_1+ k_1^2)-\alpha(\omega_2+ k_2^2)\bigg)\\[4mm]
\psi_2(\bw,\bk)^2 &=&\displaystyle \frac{1}{\beta}\bigg(\gamma_1(\omega_2+ k_2^2)-\alpha(\omega_1+ k_1^2)\bigg)\,,
\end{array}
\end{equation}
with $\beta=\gamma_1 \gamma_2- \alpha^2$, which is non-zero
due to (\ref{gamma-alpha-condition}).

The traditional approach to deriving the WMEs is to use an
averaged Lagrangian (e.g.\ Chapter 14 of \textsc{Whitham}~\cite{w74}).
The CNLS equations (\ref{cnls}) are formally the Euler-Lagrange equation for
\begin{equation}\label{vp-1}
\delta \int_{t_1}^{t_2}\int_{x_1}^{x_2} L(\Psi_t,
\Psi_x,\Psi)\,\rd x \rd t = 0\,,
\end{equation}
with $\Psi:=(\Psi_1,\Psi_2)$, fixed endpoint variations on $\delta\Psi$, and
\begin{equation}\label{Lagr-cnls}
\begin{array}{rcl}
L &=& \displaystyle\frac{\ri}{2}\left(\overline\Psi_1(\Psi_1)_t -
\Psi_1(\overline\Psi_1)_t \right) +
\frac{\ri}{2}\left(\overline\Psi_2(\Psi_2)_t -
\Psi_2(\overline\Psi_2)_t \right) \\[4mm]
&&\quad -\big|(\Psi_1)_x\big|^2
-\big|(\Psi_2)_x\big|^2 +\fr\gamma_1|\Psi_1|^4 +
\alpha |\Psi_1|^2|\Psi_2|^2 + \fr\gamma_2|\Psi_2|^4\,,
\end{array}
\end{equation}
with the overline indicating complex conjugate. The basic state
(\ref{twophase-wavetrains}) is
substituted into (\ref{Lagr-cnls}) and $L$ is averaged over the two phases
reducing it to $\mathscr{L}^{avg}(\bw,\bk)$.
This averaged Lagrangian is then assumed to depend slowly on
$X=\eps x$ and $T=\eps t$, and a secondary variational principle is
introduced
\begin{equation}\label{vp-2}
\delta \int_{T_1}^{T_2}\int_{X_1}^{X_2} \mathscr{L}^{avg}(\bw+{\bm\phi}_T,\bk+{\bm\phi}_X)\,\rd X\rd T = 0\,,
\end{equation}
obtained by replacing $\bw\mapsto \bw+{\bm\phi}_T$ and $\bk\mapsto \bk+{\bm\phi}_X$.
Taking variations with respect to ${\bm\phi}$, with
fixed endpoints, generates the vector-valued conservation of wave action
\[
\frac{\partial\ }{\partial T}\left(
\frac{\partial\mathscr{L}^{avg}}{\partial\bw}\right) +
\frac{\partial\ }{\partial X}\left(
\frac{\partial\mathscr{L}^{avg}}{\partial\bk}\right) =0\,.
\]
\red{Defining $\bq={\bm\phi}_X$ and $\bO={\bm\phi}_T$}, and adding
in the integrability condition $\partial_T\bq=\partial_X\bO$
gives the two-phase WMEs. This is the classic derivation that can be
found in \cite{ab70} and \S14.9 of \cite{w74}.

A direct approach, which is more amenable to rigorous analysis,
is to start with the exact geometric optics ansatz,
\begin{equation}\label{goptic-approx}
{\bm\Psi}(x,t) = {\bm\Psi}({\bm\theta}+\eps^{-1}{\bm\phi},\bw+\bO,\bk+\bq)+
\eps {\bf W}({\bm\theta}+\eps^{-1}{\bm\phi},X,T,\eps)\,,
\end{equation}
with
\[
  \Psi_j({\bm\theta}+\eps^{-1}{\bm\phi},\bw+\bO,\bk+\bq):=
  \re^{\ri(\theta_j+\eps^{-1}\phi_j)}  \psi_j(\bw+\bO,\bk+\bq)\,,\quad j=1,2\,,
  \]
  where the function
  ${\bm\phi}$, $\bO$, and $\bq$ are slowly
varying functions of $X=\eps x$ and $T=\eps t$, and ${\bf W}$ is a remainder term, with the constraint $\partial_T\bq=\partial_X\bO$ imposed.

The expression (\ref{goptic-approx}) is substituted into the governing
equations (\ref{cnls}) and expanded order by order in $\eps$.
Invoking a solvability condition at second order then generates
the WMEs directly
\begin{equation}\label{mp-wmes}
 \partial_T \mathscr{A}_j(\bw+\bO,\bk+\bq) + \partial_X\mathscr{B}_j(\bw+\bO,\bk+\bq)=0\,,\quad j=1,2\,,\qand
\partial_T \bq= \partial_X \bO\,.
\end{equation}
The first two equations in (\ref{mp-wmes}) are
conservation of wave action and the second two equations are imposed as a constraint based on the conservation of waves (see \textsc{Ratliff}~\cite{r17}
for details of the geometric optics approach).

We now derive the multiphase WMEs by a pure multiscale analysis, which is closer
in spirit to the forthcoming validity proof.
In this setting, the Whitham equations are obtained by setting $\eps$ to zero.
Start with the CNLS equations in (\ref{cnls}). Introduce
\[
\Psi_1 = {\rm exp}\big( r_1 + \ri (\phi_1+ \gamma_1 t+\alpha t)\big) \qand
\Psi_2 = {\rm exp}\big( r_2 + \ri (\phi_2+ \gamma_2 t+\alpha t)\big) \,.
\]
Here, $r_1,r_2,\phi_1,\phi_2$ are functions of $(x,t)$ and
$\gamma_1,\gamma_2$, $\alpha$  are the same constants as in \eqref{cnls}. For notational simplicity
 we restrict attention to the case ${\bk}=0$ and $\bw=(\gamma_1+\alpha,\gamma_2+\alpha)$ although
  the general case is completely analogous. Note also that under this restriction the background wave train
  $(r_1,\phi_1,r_2,\phi_2)\equiv0$ corresponds to  a time periodic
  solution $\Psi_i(t)=e^{\ri(\gamma_i+\alpha)t}$ which is constant in space.
  \par
Substitution into (\ref{cnls}) gives
\begin{eqnarray*}
\partial_t r_1
&=& - \partial_x^2 \phi_1 - 2(\partial_x r_1)(\partial_x \phi_1), \\
\partial_t \phi_1
&=& \partial_x^2 r_1 - (\partial_x \phi_1)^2 + (\partial_x r_1)^2 +\gamma_1 (e^{2r_1}-1) + \alpha \big(e^{2r_2}-1\big)\,,\\
\partial_t r_2
&=& - \partial_x^2 \phi_2 - 2(\partial_x r_2)(\partial_x \phi_2)\,, \\
\partial_t \phi_2
&=& \partial_x^2 r_2 - (\partial_x \phi_2)^2 + (\partial_x r_2)^2 + \gamma_2 (e^{2r_2}-1) +\alpha \big(e^{2r_1}-1\big)\,.
\end{eqnarray*}
At this point, the CNLS equation has just been transformed to polar
coordinates with moduli $(\re^{r_1},\re^{r_2})$.
To bring the equations into line with the perturbed Whitham equations
(\ref{M-wmes}),
differentiate the second and fourth equations with respect to $x$ and
introduce the new coordinates
\[
v_1 = \frac{\partial\phi_1}{\partial x} \qand
v_2 = \frac{\partial\phi_2}{\partial x} \,.
\]
Then the governing equations for $\bu:= (r_1,v_1,r_2,v_2)$ are
\begin{equation}\label{eqe12a}
  \partial_t \bu = {\bf M}(\bu)\partial_x\bu + {\bf F}(D_x^3\bu)\,,
\end{equation}
with
\begin{equation}\label{M-characteristics}
{\bf M}(\bu) =
\left[\begin{matrix} -2v_1 & -1 & 0 & 0 \\ 2\gamma_1{\rm exp}(2r_1) & -2v_1 &
    2\alpha{\rm exp}(2r_2)
      & 0 \\ 0 & 0 & -2v_2 & -1 \\ 2\alpha{\rm exp}(2r_1) & 0 & 2\gamma_2{\rm exp}(2r_2) &
    -2v_2 \end{matrix}\right] \,,
\end{equation}
and
\begin{equation}\label{F-def}
  {\bf F}(D_x^3\bu) := \begin{pmatrix} 0\\
    \partial_x^3r_1 + \partial_x(\partial_x r_1)^2 \\
0\\
\partial_x^3r_2 + \partial_x(\partial_x r_2)^2 \end{pmatrix}\,.
\end{equation}
Upon introducing scaled variables, $X=\eps x$ and $T=\eps t$, with $0\le \eps \ll 1$, and
\begin{equation}\label{ap1}
\bu(x,t) = \check{\bu}(X,T,\eps)\,,
\end{equation}
the scaled equations are
\begin{equation}\label{check-ru}
  \partial_T \check{\bu} = {\bf M}(\check{\bu})\partial_X\check{\bu} + \eps^2{\bf F}(D_X^3\check{\bu})\,.
\end{equation}
Formally taking the limit $\eps\to0$ recovers the multi-phase WMEs
\begin{equation}\label{wmes-Mform}
\partial_T\bu = {\bf M}(\bu)\partial_X\bu \,.
\end{equation}
The characteristics of (\ref{wmes-Mform})
have been studied in \textsc{Bridges \& Ratliff}~\cite{br18}
and it is found numerically that the system
can be either hyperbolic, elliptic, or mixed depending on parameter values.
The linear stability of multiphase plane waves of CNLS can be studied
directly, and for parameter
values when CNLS (\ref{cnls})
is integrable extensive linear stability results are
obtained in \textsc{Degasperis, et al.}~\cite{dls18}.
\red{However, in general the correlation between
  characteristic type in the WMEs and linear stability
  in CNLS is not one-to-one, as CNLS
 may have short wave instabilities which are missed
  by the WMEs approximation
  (cf.\ \textsc{Benzoni-Gavage, et al.}~\cite{bgnr}).}

\section{Cauchy-Kowalevskaya theory in Gevrey spaces}
\label{sec-CK-gevrey}
\setcounter{equation}{0}

In this section we prove an abstract local existence theorem
for quasilinear PDEs of the form
\begin{equation}\label{1.PDE}
\partial_T \bu={\bf M}(\bu)\partial_X \bu,\ \ \bu\big|_{T=0}=\bu_0\,.
\end{equation}
Here $X\in\R$ and $T\geq0$ are scaled variables, but since the form of the
equation
(\ref{1.PDE}) is the same in scaled and unscaled variables, the result
applies in either. The unknown vector-valued function
$\bu=(u_1(X,T),\cdots,u_d(X,T))$ is $d-$dimensional.  The matrix
${\bf M}(\bu)$ is
a given entire function and $\bu_0\in \Chi^s_\sigma(\R)$, $s>1$.
Here $\Chi_\sigma^s$ is a Gevrey space defined by the norm
\begin{equation}\label{1.gav}
\|\bu\|_{\Chi^s_\sigma}:=\|e^{\sigma(|\xi|+1)}(1+|\xi|^{2s})^{1/2}\widehat \bu(\xi)\|_{L^2}\,.
\end{equation}
We will use the same notation $\Chi^s_\sigma$ for scalar-valued and
vector-valued functions.

The required analogue of the Cauchy-Kovalevskaya theorem for
equation \eqref{1.PDE} is as follows.  Results of
  this type were first proved by \textsc{Ovsyannikov}~\cite{Ov65};
  see also \textsc{Treves}~\cite{Treves} for an elaboration of theorems
  of Cauchy-Kowalevskaya type.

\begin{theorem}\label{Th1.CK} Let $s>1$ and $\sigma_0>0$. Then, for every $R>0$, there exist $\eta=\eta(R,s,\sigma_0)$
  such that for every $\bu_0\in\Chi_{\sigma_0}^s$, with $\|\bu_0\|_{\Chi_{\sigma_0}^s}\le R$,
  there exists a unique local solution
  $\bu(T)\in \Chi^s_{\sigma(T)}$ of problem \eqref{1.PDE} with $\sigma(T):=\sigma_0-\eta T$,
  $T\in[0,\sigma_0/\eta]$, and $ \sup_{T\in[0,\sigma_0/\eta]} \| u(T) \|_{\Chi^s_{\sigma(T)}} \leq R $.
\end{theorem}

\begin{proof} We first give the formal derivation of a priori estimates
  for $\bu(T)$ in the corresponding
  Gevrey spaces. This derivation can be justified {\it a posteriori}
  in a standard way, for instance, using the vanishing viscosity method
  (which will be briefly elaborated below).

  In the proof, we will need three standard facts about Gevrey spaces:

 \begin{description}

\item{1)} $G_\sigma^s$ is an algebra for $s>\frac12$, i.e., if $u,v\in G_\sigma^s$ then $uv\in G_\sigma^s$ and

\begin{equation}\label{1.alg}
\|uv\|_{G_\sigma^s}\le C_s\|u\|_{G_\sigma^s}\|v\|_{G^s_\sigma},
\end{equation}
where the constant $C_s$ is independent of $\sigma>0$. In the vector-valued
case the product is replaced by an inner product on $\R^d$.
The formula (\ref{1.alg}) can be generalized as follows:
\begin{equation}\label{1.kp}
\|uv\|_{G^s_\sigma}\le C_s\(\|u\|_{G^s_\sigma}\|v\|_{G^1_\sigma}+\|u\|_{G^1_\sigma}\|v\|_{G^s_\sigma}\)
\end{equation}
which holds for all $s\ge0$. The exponent $1$ here can be replaced by any $\kappa>1/2$.

\item{2)} For any entire function $\phi$ such that $\phi(0)=0$ there is an entire function $\phi_s(z)$
which is positive and monotone increasing for $z\in\R_+$ and $\phi_s(0)=0$ such that
\begin{equation}\label{1.f}
\|\phi(u)\|_{G^s_\sigma}\le \phi_s(\|u\|_{G^s_\sigma}),\ \ u\in G^s_{\sigma}.
\end{equation}
\noindent
\item{3)} Let $u\in G^s_\sigma$ for some $s$ and $\sigma$. Then
   \begin{equation}\label{new-inequality}
   \|u\|_{G^{s+p}_{\sigma-\delta}}\le  C(p,\delta)\|u\|_{G^s_\sigma}\,,\quad \delta,p>0,
   \end{equation}
   where $C(p,\delta)$ is independent of $s$ and $\sigma$.
 \end{description}
 \red{
A summary of the basic properties of Gevrey spaces, in the context of parabolic PDEs, is
 given in \textsc{Ferrari \& Titi}~\cite{FT}.  We supplement those results,
 with some proofs that are needed here, in Appendix \ref{app-a}.}

Now let $A:=\sqrt{-\partial_x^2}$, and multiply equation \eqref{1.PDE}
by $e^{2\sigma(T)(1+A)}(1+A^{2s})\bu$ and integrate over $X\in\R$. After standard calculations, this gives
\begin{multline}\label{1.est}
\frac12\frac d{dT}\|\bu\|^2_{G^s_{\sigma(T)}}+\eta\|(1+A)^{1/2}\bu\|^2_{G^s_{\sigma(T)}}=
({\bf M}(\bu)\partial_X\bu,\bu)_{G^s_{\sigma(T)}}=\\
=({\bf M}(\bu)-{\bf M}(0))\partial_X\bu,\bu)_{G^s_{\sigma(T)}}+
({\bf M}(0)\partial_X\bu,\bu)_{G^s_{\sigma(T)}}\,,
\end{multline}
where $(\bu,\bv)_{G^s_{\sigma}}$ is the natural inner product in $G^s_\sigma$.
Now use the Cauchy-Schwarz inequality
$$
(\bu,\bv)_{G^s_\sigma}\le\|\bu\|_{G^{s-1/2}_\sigma}\|\bv\|_{G^{s+1/2}_\sigma}\,,
$$
together with \eqref{1.alg} and \eqref{1.f} and the assumption $s-\frac12>\frac12$, to end up with
\begin{equation*}
\frac12\frac d{dT}\|\bu\|^2_{G^s_{\sigma(T)}}+\eta\|\bu\|^2_{G^{s+1/2}_{\sigma(T)}}\le
\|{\bf M}(0)\|\|\bu\|^2_{G^{s+1/2}_{\sigma(T)}}+\phi_s(\|\bu\|_{G^{s-1/2}_{\sigma(T)}})
\|\bu\|^2_{G^{s+1/2}_{\sigma(T)}}\,,
\end{equation*}
where we have used that
\begin{multline*}
\( ({\bf M}(\bu)-{\bf M}({\bf 0}))\partial_X\bu,\bu\)_{G^s_\sigma}\le
\| ({\bf M}(\bu)-{\bf M}({\bf 0}))\partial_X\bu\|_{G^{s-1/2}_\sigma}\|\bu\|_{G^{s+1/2}_\sigma}\le\\\le
C\|{\bf M}(\bu)-{\bf M}({\bf 0})\|_{G^{s-1/2}_\sigma}\|\bu\|^2_{G^{s+1/2}_\sigma}\le
\phi_s(\|\bu\|_{G^{s-1/2}_\sigma})\|\bu\|^2_{G^{s+1/2}_\sigma}\,,
\end{multline*}
for some smooth monotone increasing function $\phi_s$ associated with ${\bf M}$.

Thus, we arrive at
\begin{equation}\label{1.good}
\frac12\frac d{dT}\|\bu\|^2_{G^s_{\sigma(T)}}+\(\eta-\|{\bf M}(0)\|-\phi_s(\|\bu\|_{G^s_{\sigma(T)}})\)
\|\bu\|^2_{G^{s+1/2}_{\sigma(T)}}\le 0\,,
\end{equation}
\red{for a possibly different $\phi_s(\cdot)$, which for notational simplicity
  we denote with the same symbol, noting that
  $\|u\|_{G^{s-1/2}_\sigma}\le C\|u\|_{G^s_\sigma}$ for some $C\geq1$.}

  \par
  Now, fix $\eta$ in such a way that $\eta>\|{\bf M}(0)\|+\phi_s(R)$. Then estimate \eqref{1.good}
guarantees that 
\begin{equation}\label{key-est}
\|\bu(T)\|_{G^s_{\sigma(T)}}\le R
\end{equation}
until $\sigma(T)$ vanishes.
Recall that the derivation of \eqref{key-est} was formal. However, if we replace $\sigma(T)$ by
$\sigma_\mu(T):=\sigma(T)-\mu$ for any positive $\mu$, all manipulations involved in the derivation of
 \eqref{key-est} become rigorous and all corresponding terms make sense if the solution satisfies
 $\bu(T)\in G^s_{\sigma(T)}$. Importantly, the obtained estimates are uniform with respect to $\mu$,
 so passing to the limit $\mu\to0$, we may justify estimate \eqref{key-est}.
 \par
 To confirm uniqueness, let $\bu_1$ and $\bu_2$ be two analytic solutions
 of \eqref{1.PDE} (belonging to $G^s_{\sigma_0}$ for all $T\in[0,T_0]$ for some $\sigma_0>0$).
 Then the difference $\bv(T):=\bu_1(T)-\bu_2(T)$ satisfies the
  analogue of equation \eqref{1.PDE}:
\begin{equation}
\partial_T\bv={\bf M}(\bu_1)\partial_X{\bf v}+({\bf M}(\bu_1)-{\bf M}(\bu_2))\partial_X\bu_2
\end{equation}
Arguing as before, but using in addition that
$$
\|{\bf M}(\bu_1)-{\bf M}(\bu_2)\|_{G^s_\sigma}\le
(\psi_s(\|\bu_1\|_{G^s_\sigma})+\psi_s(\|\bu_2\|_{G^s_\sigma}))\|\bv\|_{G^s_\sigma},
$$
where $\psi_s$ is a monotone increasing function
associated with ${\bf M}$, we
get the analogue of \eqref{1.good}:
\begin{multline}\label{1.goood}
\frac12\frac d{dT}\|\bv\|^2_{G^s_{\sigma_\mu(T)}}+
\(\eta-\|{\bf M}(0)\|-\phi_s(\|\bu_1\|_{G^s_{\sigma_\mu(T)}})\)
\|\bv\|^2_{G^{s+1/2}_{\sigma_\mu(T)}}\le\\\le  C_s\|\bu_2\|_{G^{s+1}_{\sigma_\mu(T)}}
\(\psi_s(\|\bu_1\|_{G^s_{\sigma_\mu(T)}})+
\psi_s(\|\bu_2\|_{G^s_{\sigma_\mu(T)}})\)\|\bv\|_{G^s_{\sigma_\mu(T)}}^2\,,
\end{multline}
for every $\mu>0$. Taking the exponent $\eta$ large enough, estimating the term
$\|\bu_2\|_{G^{s+1}_{\sigma_\mu(T)}}$ with the help of \eqref{new-inequality} and applying the Gronwall inequality
 to \eqref{1.goood}, we get the desired uniqueness on a possibly smaller interval $T\in[0,T_{00}]$ determined by the
 assumption that $\sigma_\mu(T)>0$. Repeating these arguments on the next time intervals, we get the uniqueness
 on the whole existence interval $[0,T_0]$.
\par
Let us now discuss the role of the vanishing viscosity method in
the existence of a solution.
Let $\beta$ be a small positive parameter and consider the following semilinear
parabolic equation
$$
\partial_T\bu^\beta={\bf M}(\bu^\beta)\partial_X\bu^\beta+\beta\partial^2_X\bu^\beta,\ \ \bu^\beta\big|_{T=0}=\bu_0
$$
Since the operator $A:=\beta\partial_X^2$ generates an analytic
 semigroup in Gevrey spaces $G^s_{\sigma_0}$, the local solution  $\bu^\beta$ can be constructed, say, via the
 contraction mapping principle or the
 implicit function theorem (arguments like this, using the theory of
 analytic semigroups, can be found in \textsc{Henry}~\cite{Henry}).
 Crucial for the method is that all of the
 above estimates are {\it uniform} with respect to $\beta$ (the extra term
 $\beta\|\bu^\beta\|^2_{G^s_{\sigma(T)}}$ is
 non-negative and does not impact the estimates). Thus, we have \eqref{key-est}
  for $\bu^\beta(T)$ which is uniform with respect to $\beta$. Passing to the limit $\beta\to0$,
  we get the desired solution $\bu$ of \eqref{1.PDE}.  This
  completes the proof of the theorem.
\end{proof}

\begin{remark} Arguing analogously, it is possible to obtain a similar result
  for the slightly more general equation:
\begin{equation}\label{1.eqfg}
\Dt \bu={\bf M}(\bu)\partial_X\bu+{\bf f}(\bu),\ \ \bu\big|_{T=0}=\bu_0,
\end{equation}
where ${\bf f}(\bu)$ is another entire function satisfying ${\bf f}(0)=0$.
\end{remark}

\subsection{Local existence for the multiphase WMEs}
\label{subsec-MPWMEs-existence}

Local existence for the multiphase WMEs now follows by applying
Theorem \ref{Th1.CK} to the system (\ref{wmes-Mform}).
The matrix ${\bf M}(\bu)$ in (\ref{M-characteristics})
is an entire function and $d=4$.  Hence, regardless of whether
the multiphase WMEs are hyperbolic, elliptic or of mixed type,
local existence will follow in Gevrey spaces.
\vspace{.15cm}

\begin{theorem}\label{theoremA}
  Let $s>1$ and $ \sigma_0 > 0 $, and let
  \[
  \bu^*(X,T)\Big|_{T=0} = \bu_0(X):=(r_{10}^*(X),v_{10}^*(X),r_{20}^*(X),v_{20}^*(X))\,.
   \]
  Then there exist
$T_0>0$ and $ C > 0 $ such that
for all initial data $\bu^*(X,0)\in \Chi^s_{2\sigma_0}$
with  $\| \bu^*(X,0)\|_{\Chi^{s}_{2\sigma_0}} \leq C $,
the multiphase WMEs (\ref{wmes-Mform})
have a unique solution
\[
\bu^*(X,T)\in C\left([0,T_0],\Chi^{s}_{\sigma_0}\right)\quad\mbox{with}
\quad \bu^*(X,T)\Big|_{T=0}=\bu_0\,.
\]
\end{theorem}

\section{Approximate solutions for the perturbed problem}
\label{sec-approx}
\setcounter{equation}{0}

In this section, a continuation of the abstract theory of
\S\ref{sec-CK-gevrey} is given for the case when there is a perturbation
of the quasilinear system (\ref{1.PDE}) of the form
\begin{equation}\label{2.PDE}
\Dt \bu={\bf M}(\bu)\partial_X \bu+\eb {\bf F}(D^k_X\bu),\ \ \bu\big|_{T=0}=\bu_0\,,
\end{equation}
where $D^k_X \bu=\{\bu,\partial_X \bu,\cdots,\partial^k_X \bu\}$, $\eb\in\R$ is a small parameter and ${\bf F}$ is a
given entire function satisfying ${\bf F}(0)=0$.  In this section $X\in \R$
and $T\geq 0$ are scaled variables.  Here $\eb$ is an arbitrary small parameter but will be restored to $\eb=\eps^2$ when the theory is applied to
the system (\ref{u-M-form}).

We seek an approximation to the solution $\bu$ in the form
 \begin{equation}\label{2.taylor}
\bu(X,T,\nu)=\bu^0(X,T)+\eb \bu^1(X,T)+\eb^2\bu^2(X,T)+\cdots
 \end{equation}
Inserting these expansions into equation \eqref{2.PDE} and equating the terms with the
same powers of $\eb$, we get at $\eb^0$,
\begin{equation}\label{2.PDE0}
\Dt \bu^0={\bf M}(\bu^0)\partial_X \bu^0\,,\quad\mbox{with}\ \bu^0\big|_{T=0}=\bu_0\,,
\end{equation}
which coincides with equation \eqref{1.PDE} studied earlier.
The higher-order terms $\bu^n$,
$n\ge1$ can be found by solving the inhomogeneous equations of variation
associated with problem \eqref{2.PDE0}:
\begin{equation}\label{2.app}
\Dt \bu^n-{\bf M}(\bu^0)\partial_X \bu^n-\red{D_{\bf u}}{\bf M}(\bu^0)\bu^n\partial_X \bu^0={\bf F}_n(\bu^0,\cdots,\bu^{n-1}),\ \mbox{with}\ \bu^n\big|_{T=0}=0,
\end{equation}
where ${\bf F}_n$ is an entire function of the lower order approximations
$\bu^0,\cdots,\bu^{n-1}$ and their
derivatives \red{up to order $k$} satisfying ${\bf F}_n(0)=0$.
Multiplying this equation
 by $e^{2\sigma(T)(A+1)}(1+A^{2s})\bu^n(T)$ and
arguing as in the proof of Theorem \ref{Th1.CK}, we arrive at
\begin{multline}
\frac12\frac d{dT}\|\bu^n\|^2_{G^s_{\sigma(T)}}+
\(\eta-\phi_s(\|\bu^0\|_{G^{s}_{\sigma(T)}})\)
\|\bu^n\|^2_{G^{s+1/2}_{\sigma(T)}}\le\\ \le
\psi_s\Big( \|\bu^0\|_{G^{s+1}_\sigma}\Big)\|\bu^n\|_{G^s_\sigma}^2+C\|{\bf F}_n\|_{G^s_{\sigma(T)}}^2,
\end{multline}
where $\phi_s$ and $\psi_s$ are real analytic monotone increasing functions depending only on ${\bf M}$ and $s$.
 We see that the solvability condition

 \begin{equation}\label{2.solv}
\eta-\phi_s(\|\bu^0\|_{G^{s}_{\sigma(T)}})>0,\ \ T<\sigma_0/\eta
 \end{equation}
 for this inequality depends only on the initial local solution $\bu^0(T)$, so the lifespan
  of every of $\bu^n$ is determined by the properties of $\bu^0(T)$ only. Indeed, if \eqref{2.solv}
   is satisfied, we get the recursive estimate
   \begin{equation}\label{2.rec}
   \|\bu^n(T)\|_{G^s_{\sigma(T)}}^2\le C e^{KT}
   \sup_{0\le\tau\le T}\|{\bf F}_n\big(\bu^0(\tau),\cdots,\bu^{n-1}(\tau)\big)\|_{G^s_{\sigma(\tau)}}^2\,,
   \end{equation}
for some positive constants $C$ and $K$ depending on $\bu^0$. Since ${\bf F}_n$ depends only
on $\bu^0, \cdots, \bu^{n-1}$, estimate \eqref{2.rec} allows us to construct
the correctors $\bu^n$ recursively solving the linear equation
 \eqref{2.app} at every step (the existence of a solution for problem \eqref{2.app} can be proved
  exactly as in Theorem \ref{Th1.CK}). Indeed, it follows from the vector-valued version of
  \eqref{1.f} that
\begin{equation}\label{2.rrec}
  \|{\bf F}_n\|_{G^s_{\sigma}}\le F_s\(\|\bu^0\|_{G^{s+k}_\sigma}+\cdots+\|\bu^{n-1}\|_{G^{s+k}_{\sigma}}\)
\end{equation}
for $s>\frac12$, $\sigma\ge0$ and some smooth monotone increasing function $F_s$ with $F_s(0)=0$.

   However, there is still a small problem
   with iterating inequalities \eqref{2.rec} and \eqref{2.rrec}, namely, the number
    of derivatives of $\bu^0$ which we need to control in order to estimate ${\bf F}_n$
   grows with $n$, so in order to get the $G^s_\sigma$-norm of ${\bf F}_n$, we need at
   least a $G^{s+kn}_\sigma$-norm
   of $\bu^0$. Therefore, performing iterations in a straightforward way will  require
    the solvability condition \eqref{2.solv} to be satisfied not only for $s$, but also for
    $s+k$, $s+2k$ and so on.
    Since $\phi_s$ is growing in $s$, we will have to take $\eta$ depending on $n$ and shrink the
    existence interval for the correctors when $n$ is growing.     To overcome this and to
    get the  lifespan uniform with respect to $n$ for the correctors $\bu^n$, we will
     proceed in
     an alternative way avoiding  increasing the exponent $s$, but decreasing slightly the analyticity
     exponent $\sigma$ and utilizing \eqref{new-inequality}.
   This gives the following result.
   \begin{theorem}\label{Th2.sol} Let $s>1$ and let the local solution $\bu^0(T)$ of equation
    \eqref{2.PDE} satisfying \eqref{2.solv}. Then, for every $\delta>0$, the correctors $\bu^n(T)$
    satisfy the estimate
    \begin{equation}\label{2.corest}
    \|\bu^n(T)\|_{G^s_{\sigma(T)-\delta}}\le Q_{\delta,n}
    \(\sup_{0\le \tau\le T}\|\bu^0(\tau)\|_{G^{s}_{\sigma(\tau)}}\),\ \ T\le (\sigma_0{\color{black}-\delta})/\eta
    \end{equation}
     for some monotone increasing function $Q_{\delta,n}$.
   \end{theorem}

   \begin{proof} {\color{black} Let $0<\delta<\sigma_0$ be arbitrary.
   From estimates \eqref{2.rrec} and \eqref{new-inequality}, we get
\begin{equation}\label{2.rrrec}
  \|{\bf F}_n\|_{G^s_{\sigma-\delta/2^n}}\le
   F_{s,n}\(\|\bu^0\|_{G^{s}_\sigma}+\cdots+\|\bu^{n-1}\|_{G^{s}_{\sigma}}\)
\end{equation}
for some monotone increasing function $F_{s,n}$ which is independent of $\sigma$ and $\delta$. Combining this estimate with \eqref{2.rec},
and introducing $\delta_n:=\delta\sum_{l=1}^n2^{-l}$, we have
$$
\|\bu^n(T)\|_{G^s_{\sigma(T)-\delta_n}}\le
 CF_{s,n}\(\sup_{0\le\tau\le T}\|\bu^0(\tau)\|_{G^s_{\sigma(\tau)-\delta_{n-1}}}+
\cdots+\sup_{0\le\tau\le T}\|\bu^{n-1}(\tau)\|_{G^s_{\sigma(\tau)-\delta_{n-1}}}\)\,.
$$
}
Iterating this estimate, we finally arrive at \eqref{2.corest}.
\end{proof}

Now introduce the $n$-th order approximations
\begin{equation}
\widetilde \bu^n(T):=\bu^0(T)+\eb \bu^1(T)+\cdots +\eb^n \bu^n(T)\,,
\end{equation}
and the corresponding residuals
\begin{equation}\label{2.res}
\textsf{Res}^n(T):=\Dt \widetilde \bu^n(T)-{\bf M}(\widetilde \bu^n(T))\partial_X \widetilde \bu^n(T)-\eb {\bf F}(D^k\widetilde \bu^n(T))\,.
\end{equation}
{\color{black} Since the equations for the correctors $\bu^l$, $l=0,\cdots,n$
are obtained by equating to zero the coefficients in front of $\nu^l$ in Taylor expansions
 of \eqref{2.res} in $\nu$, we have
\begin{equation}
\textsf{Res}^n(T)=\nu^{n+1}\Phi_n(\nu,D^k\bu^0(T),\cdots,D^k\bu^n(T))
\end{equation}
for some function $\Phi_n$ which is analytic with respect to all arguments.}
Applying the theorem {\color{black} and estimate \eqref{1.f}} then gives the following.
\begin{corollary}\label{Cor2.res} Let $\bu^0(T)$, $T<\sigma_0/\eta$ be the local solution of equation
 \eqref{2.PDE} satisfying the condition \eqref{2.solv} for some $s>1$. Then, for every $\delta>0$,
  the approximate solutions $\widetilde \bu^n(T)$ and residuals $\textsf{Res}^n(T)$ satisfy:
  \begin{multline}
\|\widetilde \bu^n(T)\|_{G^s_{\sigma(T)-\delta}}\le \Phi_{n,\delta}\(
{\color{black}\sup_{0\le\tau\le T}\|\bu^0(\tau)\|_{G^s_{\sigma(\tau)}}}\)\qand\\
\|\textsf{Res}^n(T)\|_{G^s_{\sigma(T)-\delta}}\le \eb^{n+1}\Phi_{n,\delta}
\({\color{black}\sup_{0\le\tau\le T}\|\bu^0(\tau)\|_{G^s_{\sigma(\tau)}}}\)
  \end{multline}
 for $T<(\sigma_0-\delta)/\eta$
 {\color{black} and some monotone increasing function  $\Phi_{n,\delta}$ depending on $n,\delta$
  and $s$}.
\end{corollary}

\begin{remark} Note that Theorem \ref{Th2.sol} and Corollary \ref{Cor2.res} are proved under
the assumption that there exists a local solution $\bu^0(T)$, $T<\sigma_0/\eta$, which satisfies
condition \eqref{2.solv}. In this case the lifespan of the approximate solution $\widetilde \bu^n(T)$
remains the same as the lifespan $T=\sigma_0/\eta$ no matter how big $n$ is. However, if we just
have an analytic local solution $\bu^0(T)$ of equation \eqref{2.PDE} defined on the interval
$T\in[0,T_0]$ such that
$$
\|\bu^0(T)\|_{G^{s}_{\sigma_0-\eta' T}}\le R,\ \  T_0\le \sigma_0/\eta'\,,
$$
for some positive $\sigma_0$ and non-negative $\eta'$, then the key assumption \eqref{2.solv} is
not automatically satisfied. In order to satisfy it we need to increase $\eta'$ till
$$
\eta=\phi_s(R)
$$
and this decreases the lifespan of the approximate solution $\widetilde \bu^n(T)$ till $T_1:=\sigma_0/\phi_s(R)$.
Thus, in this general situation we can only guarantee the existence of analytic
approximate solutions $\widetilde \bu^n(T)$ on a smaller interval $t\in[0,T_1]$ than the initial lifespan $T_0$
of the solution $\bu^0(T)$.
\end{remark}

The finite Taylor expansions, proved to exist in this
section, generate higher-order
corrections to multiphase Whitham modulation theory.
Heretofore higher order WMEs
have only been studied in the single phase case (e.g.\ \textsc{Luke}~\cite{luke66} and \S2 of \textsc{D\"ull \& Schneider}~\cite{ds09}).
Although these expansions have uniform lifespan with respect to $n$,
 the series \eqref{2.taylor} are usually divergent and do not give the exact solution of the perturbed problem
 no matter how small $\eb$ is. In any case, the
 higher-order corrections from this section are useful for
 generating higher order
 in $\eb$ estimates on the residuals (see Theorem \ref{thrm7.1} below).
 We will take up filling the gap between $\widetilde{\bu^n}$ and the exact
 solution in the next section.

\section{Exact solutions in Gevrey spaces}
\label{sec-exactsolutions}
\setcounter{equation}{0}

We now look at the validity question: how well do solutions of
the multiphase WMEs approximate the solutions of CNLS.
The starting point is the exact equations for CNLS in
$(r_j,v_j)$ coordinates rewritten here as
\begin{equation}\label{3.WME}
  \partial_T \bu = {\bf M}(\bu)\partial_X\bu + \nu{\bf F}(D_X^3\bu)\,,\quad \bu\Big|_{T=0}=\bu_0\,,
\end{equation}
with ${\bf M}(\bu)$ defined in (\ref{M-characteristics}) and ${\bf F}$
defined in (\ref{F-def}).
The independent variables $X,T$ are scaled variables, and for
notational convenience we have dropped the circumflex on $\bu$ and
used $\nu$ as the small parameter.   Translation to other notations
will be straightforward {\it a posteriori}.

This system has the form of \eqref{2.PDE}
with $d=4$ and $k=3$, and so the results generated in the previous sections
will carry over. In particular,
the limit system \eqref{3.WME} with $\eb=0$ (which is exactly the WMEs) has a unique local analytic
solution $\bu^0(T)$ and, when $\nu\neq0$, the approximate solutions
$\widetilde \bu^n(T)$ are well-defined and satisfy the
 estimates of Theorem \ref{Th2.sol} and Corollary \ref{Cor2.res}.

 We seek the desired exact solution of (\ref{3.WME}) in the form
 \begin{equation}\label{3.an}
\bu=\widetilde \bu^n+\bv :=(\widetilde r_1^n+R_1,\widetilde v_1^n+V_1,\widetilde r_2^n+R_2,\widetilde v_2^n+V_2)\,.
 \end{equation}
 For notational convenience, the superscript $n$ will be dropped on $\widetilde\bu^n$.
 Inserting this ansatz into equations \eqref{3.WME}, we end up with the
 following equations,
 \begin{equation}\label{3.error}
\begin{array}{rcl}
\Dt R_1 &=&-\partial_X V_1-2\widetilde v_1\partial_X R_1-2V_1\partial_X\widetilde r_1-2V_1\partial_X R_1+\textsf{Res}_1^n,\\[2mm]
\Dt V_1 &=&-\partial_X(V_1)^2-2\partial_X(\widetilde v_1 V_1)
+
2\gamma_1\partial_X\(e^{2\widetilde r_1}(e^{2R_1}\!-\!1)\)\\[2mm]
&&\quad +
2\alpha\partial_X\(e^{2\widetilde r_2}(e^{2R_2}\!-\!1)\)
+\eb\partial_X^3R_1+\eb\partial_X(\partial_X R_1)^2\\[2mm]
&&\quad\quad +2\eb\partial_X(\partial_X\widetilde r_1\partial_XR_1)+\textsf{Res}_2^n;\\[2mm]
\Dt R_2 &=&-\partial_X V_2-2\widetilde v_2\partial_X R_2-2V_2\partial_X\widetilde r_2-2V_2\partial_X R_2+\textsf{Res}_3^n,\\[2mm]
\Dt V_2 &=&-\partial_X(V_2)^2-2\partial_X(\widetilde v_2 V_2)
+ 2\gamma_2\partial_X\(e^{2\widetilde r_2}(e^{2R_2}\!-\!1)\)\\[2mm]
&&\quad +
2\alpha\partial_X\(e^{2\widetilde r_1}(e^{2R_1}\!-\!1)\)
+\eb\partial_X^3R_2+\eb\partial_X(\partial_X R_2)^2\\[2mm]
&&\quad\quad
+2\eb\partial_X(\partial_X\widetilde r_2\partial_XR_2)+\textsf{Res}_4^n\,,
\end{array}
 \end{equation}
 {\color{black} where the residuals ${\bf Res}^n:=
  (\textsf{Res}_1^n,\textsf{Res}_2^n,\textsf{Res}_3^n,\textsf{Res}_4^n)$
are defined by \eqref{2.res}.}
\par
 These equations are endowed with zero initial conditions.
 Our task is now to verify that they have a unique analytic
 local solution which is of order
 $\eb^{n+1}$. This will be our next theorem.
\begin{theorem}\label{thrm7.1} Suppose the approximate solution $\widetilde \bu^n(T)\in G^{s+1/2}_{\sigma(T)}$ of problem \eqref{3.WME}
satisfies the analogue of \eqref{2.solv}
\begin{equation}\label{3.solv}
\eta-\phi\(\|\widetilde \bu^n\|_{G^{s+1/2}_{\sigma(T)}}\)>0
\end{equation}
for the properly chosen constant $\eta>0$, $s>2$, $\sigma_0>0$, smooth monotone function $\phi$ (depending on $s$) and all $T<\sigma_0/\eta$.
Assume also that the residual satisfies
\begin{equation}\label{3.ress}
\|\textbf{Res}^n(T)\|_{G^{s}_{\sigma(T)}}^2\le Q_n(\widetilde \bu^0)\eb^{2(n+1)}\,.
\end{equation}
  Then, for sufficiently small $\eb>0$, system \eqref{3.error} possesses a
  unique solution $\bv(T)\in G^{s-1}_{\sigma(T)}$, \red{satisfying $\bv(0)=0$},
  and the following estimate holds:
 \begin{equation}\label{key-v}
 \|\bv(T)\|^2_{E^s_{\sigma(T)}}  \le  C\eb^{2(n+1)}Q_{n}(\widetilde \bu^0)\,,
\end{equation}
for all $T<\sigma_0/\eta$, where
\begin{equation}\label{Enorm-def}
\|\bv(T)\|^2_{E^s_{\sigma(T)}} :=
  \|\bv(T)\|^2_{G^{s-1}_{\sigma(T)}}+
  \eb\(\|R_1(T)\|^2_{G^{s}_{\sigma(T)}}+\|R_2(T)\|_{G^{s}_{\sigma(T)}}^2\)\,.
\end{equation}
\end{theorem}


\begin{proof} {\color{black} As in the proof of Theorem \ref{Th1.CK}, we start with a formal
 derivation of the key estimate \eqref{key-v}.}
   Take a scalar product of
the second equation in (\ref{3.error}) with $(1+A^{2(s-1)})V_1$ in the space $G^0_{\sigma(T)}$. This gives
\begin{multline}
\frac12\frac d{dT}\|V_1\|^2_{G^{s-1}_{\sigma(T)}}+
\eta\|(1+A)^{1/2}(1+A^{2(s-1)})^{1/2}V_1\|_{G^0_{\sigma(T)}}^2\\
=\(-\partial_X(V_1)^2,(1+A^{2(s-1)})V_1\)_{G^0_{\sigma(T)}}-
2\(\partial_X(\widetilde v_1 V_1),(1+A^{2(s-1)})V_1\)_{G^0_{\sigma(T)}}\\+
2\gamma_1\(\partial_X\(e^{2\widetilde r_1}(e^{2R_1}\!-\!1)\),(1+A^{2(s-1)})V_1\)_{G^0_{\sigma(T)}}\\+
2\alpha\(\partial_X\(e^{2\widetilde r_2}(e^{2R_2}\!-\!1)\),(1+A^{2(s-1)})V_1\)_{G^0_{\sigma(T)}}\\+
\eb\(\partial_X^3R_1,(1+A^{2(s-1)})V_1\)_{G^0_{\sigma(T)}}+
\eb\(\partial_X(\partial_X R_1)^2,(1+A^{2(s-1)})V_1\)_{G^0_{\sigma(T)}}\\
+2\eb\(\partial_X(\partial_X\widetilde r_1\partial_XR_1),(1+A^{2(s-1)})V_1\)_{G^0_{\sigma(T)}}+
\(\textsf{Res}_2^n,(1+A^{2(s-1)})V_1\)_{G^0_{\sigma(T)}}.
\end{multline}
The scalar products containing only first order derivatives in $X$ (in the left entry) can be estimated
 exactly as in the proof of Theorem \ref{Th1.CK}. For instance,
\begin{multline}
\left|\(\partial_X\(e^{2\widetilde r_1}(e^{2R_1}\!-\!1)\),(1+A^{2(s-1)})V_1\)_{G^0_{\sigma(T)}}\right|\\\le
\phi_s\(\|\widetilde r_1\|_{G^{s-1/2}_{\sigma(T)}}\)\phi_s\(\|R_1\|_{G^{s-1}_{\sigma(T)}}\)
\|R_1\|_{G^{s-1/2}_{\sigma(T)}}\|V_1\|_{G^{s-1/2}_{\sigma(T)}}\\
\le \phi_s\(\|\widetilde \bu^n\|_{G^{s-1/2}_{\sigma(T)}}\)
\phi_s\(\|\bv\|_{E^s_{\sigma(T)}}\)\|\bv\|_{G^{s-1/2}_{\sigma(T)}}^2\,.
 \end{multline}
The remaining first order terms are estimated analogously,
\begin{multline}
\frac12\frac d{dT}\|V_1\|^2_{G^{s-1}_{\sigma(T)}}+
(\eta-1)\|V_1\|_{G^{s-1/2}_{\sigma(T)}}^2\\\le Q_n(\widetilde \bu^0)\eb^{2(n+1)}+
\phi_s\(\|\widetilde \bu^n\|_{G^{s-1/2}_{\sigma(T)}}\)
\phi_s\(\|\bv\|_{E^s_{\sigma(T)}}\)\|\bv\|_{G^{s-1/2}_{\sigma(T)}}^2
+\eb\(\partial_X^3R_1,A^{2(s-1)})V_1\)_{G^0_{\sigma(T)}}\!\!\!\\+
\eb\(\partial_X((\partial_X R_1)^2+\partial_X\widetilde r_1\partial_XR_1),
A^{2(s-1)})V_1\)_{G^0_{\sigma(T)}}.
\end{multline}
Performing the same action with the remaining equations of \eqref{3.error} and taking a
sum of the obtained estimates, we arrive at
\begin{multline}\label{3.huge1}
\frac12\frac d{dT}\|\bv\|^2_{G^{s-1}_{\sigma(T)}}+
\(\eta-1-\phi_s(\|\widetilde \bu^n\|_{G^{s-1/2}_{\sigma(T)}})
\phi_s(\|\bv\|_{E^s_{\sigma(T)}})\)\|\bv\|_{G^{s-1/2}_{\sigma(T)}}^2 \\\le  Q_n(\widetilde \bu^0)\eb^{2(n+1)}+\!\!
\eb\!\(\partial_X^3R_1,A^{2(s-1)})V_1\)_{G^0_{\sigma(T)}}\!\!\!\!\!+\!
\eb\!\(\partial_X((\partial_X R_1)^2\!\!\!+\!2\partial_X\widetilde r_1\partial_XR_1),
A^{2(s-1)})V_1\)_{G^0_{\sigma(T)}}\!\!\!\\+
\eb\!\(\partial_X^3R_2,A^{2(s-1)})V_2\)_{G^0_{\sigma(T)}}\!\!+
\eb\!\(\partial_X((\partial_X R_2)^2\!\!+2\partial_X\widetilde r_2\partial_XR_2),A^{2(s-1)})V_2\)_{G^0_{\sigma(T)}}.
\end{multline}
In contrast to the first order terms, estimates for higher order terms
are a bit more delicate and can not
be preformed on the level of $G^{s-1}_{\sigma}$-norms. To handle them
we use the special structure of equations
  \eqref{3.WME}, namely, the possibility to get anisotropic estimates where the $R_i$ components are taken in $G^s_\sigma$-norms and
  $V_i$ components remain in the $G^{s-1}_\sigma$-norm.
  With this strategy, the estimates for the second order terms in
  \eqref{3.huge1}
  also become straightforward. For instance, using the Cauchy-Schwarz
  inequality, we get
  \begin{equation}
  \begin{array}{rcl}\label{3.huge}
\displaystyle\eb\!\(\partial_X((\partial_X R_1)^2,
A^{2(s-1)})V_1\)_{G^0_{\sigma(T)}} &\le&\displaystyle
2\eb\|\partial_X R_1\partial^2_x R_1\|_{G^{s-3/2}_\sigma}\|V_1\|_{G^{s-1/2}_\sigma}\\[4mm]
  &&\displaystyle \le
C\eb^{1/2}\|R_1\|_{G^{s}_\sigma}\(\eb\|R_1\|^2_{G^{s+1/2}_\sigma}+\|V_1\|^2_{G^{s-1/2}_\sigma}\)\,.
  \end{array}
  \end{equation}
  Estimating the other second-order terms in \eqref{3.huge1} analogously, we arrive at
\begin{multline}\label{3.huge2}
\frac12\frac d{dT}\|\bv\|^2_{G^{s-1}_{\sigma(T)}}+
\(\eta-1-\phi_s(\|\widetilde \bu^n\|_{G^{s-1/2}_{\sigma(T)}})
\phi_s(\|\bv\|_{E^s_{\sigma(T)}})\)\|\bv\|_{G^{s-1/2}_{\sigma(T)}}^2\\\le  Q_n(\widetilde \bu^0)\eb^{2(n+1)}+\!\!
\eb\!\(\partial_X^3R_1,A^{2(s-1)})V_1\)_{G^0_{\sigma(T)}}+
\eb\!\(\partial_X^3R_2,A^{2(s-1)})V_2\)_{G^0_{\sigma(T)}}\\+
C\(\|\bv\|_{E^s_{\sigma}}+\|\widetilde \bu^n\|_{E^{s+1/2}_\sigma}\)\|\bv\|^2_{E^{s+1/2}_\sigma},
\end{multline}
where the constant $C$ is independent of $\eb$.

To complete this estimate, we need to analyze the $G^s_\sigma$-norm of
the solutions $R_1$ and $R_2$. To this end,
we multiply the first equation of \eqref{3.error} by
$$
\eb(1+A^{2s})e^{2\sigma(T)A}R_1=\eb(1-\partial^2_X A^{2(s-1)})e^{2\sigma(T)A}R_1
$$
and integrate over $X$. This gives
\begin{multline}\label{3.huge-r}
\eb\frac12\frac d{dT}\|R_1\|^2_{G^s_{\sigma(T)}}+
\eta\eb\|R_1\|^2_{G^{s+1/2}_{\sigma(T)}}=
\eb(\partial_X V_1,\partial^2_{X}A^{2(s-1)}R_1)_{G^0_{\sigma(T)}}-\eb\(\partial_XV_1,R_1\)_{G^0_{\sigma(T)}}\\-
\eb\(2\widetilde v_1\partial_X R_1+2V_1\partial_X\widetilde r_1+2V_1\partial_X R_1,(1+A^{2s})R_1\)_{G^0_{\sigma(T)}}\!\!\!+
\eb(\textsf{Res}_1^n,(1+A^{2s})R_1)_{G^0_{\sigma(T)}}.
\end{multline}
Using the Cauchy-Schwarz inequality together with the estimate \eqref{1.kp}, we arrive at
\begin{multline}\label{3.huge-ru}
 2\eb(V_1\partial_X R_1,(1+A^{2s})R_1)_{G^0_{\sigma(T)}}\le 2\eb\|V_1\partial_X R_1\|_{G^{s-1/2}_{\sigma(T)}}\|R_1\|_{G^{s+1/2}_{\sigma(T)}}  \\
  \le 2C\eb(\|V_1\|_{G^1_{\sigma(T)}}\|R_1\|_{G^{s+1/2}_{\sigma(T)}} + \|V_1\|_{G^{s-1/2}_\sigma(T)}\|R_1\|_{G^2_{\sigma(T)}})\|R_1\|_{G^{s+1/2}_{\sigma(T)}}\\
  \le 2C\eb\|V_1\|_{G^1_{\sigma(T)}}\|R_1\|^2_{G^{s+1/2}_{\sigma(T)}} + C\eb^{1/2}\|R_1\|_{G^{2}_{\sigma(T)}}(\|V_1\|^2_{G^{s-1/2}_{\sigma(T)}} + C\eb\|R_1\|^2_{G^{s+1/2}_{\sigma(T)}}) \\ \le C\|\bv\|_{E^s_{\sigma(T)}}\|\bv\|^2_{E^{s+1/2}_{\sigma(T)}}.
\end{multline}
Estimates for all others terms on the right-hand side of \eqref{3.huge-r} except for the first one can be performed exactly as in
\eqref{3.huge}. This gives
\begin{multline}\label{3.huge-r1}
\eb\frac12\frac d{dT}\|R_1\|^2_{G^s_{\sigma(T)}}+ \eb(\eta -1)\|R_1\|^2_{G^{s+1/2}_{\sigma(T)}} \le
\eb(\partial_X V_1,\partial^2_{X}A^{2(s-1)}R_1)_{G^0_{\sigma(T)}}
 \\ +C\(\|\widetilde \bu^n\|_{E^{s+1/2}_{\sigma(T)}}+\|\bv\|_{E^s_{\sigma(T)}}\)\|\bv\|^2_{E^{s+1/2}_{\sigma(T)}} + Q_n(\widetilde \bu^0)\eb^{2(n+1)}.
\end{multline}
Note finally that due to integration by parts
$$
\(\partial_X V_1,\partial^2_{X}A^{2(s-1)}R_1\)_{G^0_{\sigma(T)}}+
\(\partial_X^3R_1,A^{2(s-1)})V_1\)_{G^0_{\sigma(T)}}=0.
$$
With this identity, the first term in \eqref{3.huge-r1} will be cancelled
after the summation with equation \eqref{3.huge1}.
Performing the analogous estimates with the equation for $R_2$,
and taking a sum, we cancel all third order terms in the estimate \eqref{3.huge1} and arrive  at the desired estimate
\[
\begin{array}{rcl}
&&\displaystyle\frac12\frac d{dT}\|\bv\|_{E^s_{\sigma(T)}}+
  \(\eta-C-\phi_s\(\|\widetilde \bu^n\|_{E^{s+1/2}_{\sigma(T)}}\)\phi_s\(\|\bv\|_{E^s_{\sigma(T)}}\)\)\|\bv\|^2_{E^{s+1/2}_{\sigma(T)}}\\[4mm]
  &&\hspace{2.5cm}\displaystyle \le C\eb^{2(n+1)}Q_n(\widetilde \bu^0)
  \,,\quad \bv\big|_{T=0}= 0\,,
\end{array}
\]
for some monotone smooth function $\phi_s$ and positive constant $C$ which are independent of $\eb$.
This gives the desired estimate  \eqref{3.ress} if the constant $\eta$ is large enough so that
$$
\eta-C-\phi_s\(\|\widetilde \bu^n\|_{E^{s+1/2}_{\sigma(T)}}\)\phi_s\(\|\bv\|_{E^s_{\sigma(T)}}\)>0,\ \ T<\sigma_0/\eta.
$$
In this case we will have $\|\bv\|_{E^s_{\sigma(T)}}=O(\eb^{n+1})$.
By this reason, the condition will
be satisfied for sufficiently small $\eb$ if, say,
$$
\eta-C-2\phi_s(0)\phi_s\(\|\widetilde \bu^n\|_{E^{s+1/2}_{\sigma(T)}}\):=\eta-\phi\(\|\widetilde \bu^n\|_{E^{s+1/2}_{\sigma(T)}}\)>0\,.
$$
{\color{black} This estimate finishes the derivation of the desired estimate \eqref{key-v}. Exactly as in Theorem \ref{Th1.CK}, the above
derivation becomes rigorous if we replace $\sigma(T)$ by $\sigma_\mu(T):=\sigma(T)-\mu$
 for any small positive $\mu$. Passing after that to the limit $\mu\to0$ we get the desired
 estimate for $\mu=0$ as well.
\par
The uniqueness of the solution $\bv$ in the class of analytic functions can be obtained similarly to the
 proof of Theorem \ref{Th1.CK} by writing out the equation for difference of two solutions and deriving
 the analogue of estimate \eqref{key-v} for the obtained equation.
 Since this proof follows
  almost word for word the derivation given above, we leave
   it to the reader.
\par
Finally, the existence of a solution $\bv$ can be obtained using the vanishing viscosity
method (also similar to the proof of Theorem \ref{Th1.CK}). The only difference here is that
adding the viscosity of the form $-\beta\partial_X^2\bv$ is {\it not sufficient} to
make extended equation parabolic and semilinear
(due to the presence of third order dispersion terms), so the
 viscosity term should be replaced by $\beta\partial_X^4{\bf v}$. Then the analogue of estimate \eqref{key-v}
 will obviously be uniform with respect to $\beta$ and passage to the limit $\beta\to0$ gives the
  desired solution. Thus, the theorem is proved.
}
\end{proof}

\subsection{Summary of CNLS to WMEs reduction}
\label{sec-summary}
\setcounter{equation}{0}

The main approximation theorem for (\ref{check-ru}) is as follows,
stated in terms of $\eps$ using $\nu=\eps^2$.
\vspace{.15cm}

\begin{theorem}
\label{TheoremC} Let $(r_1^*,v_1^*,r_2^*,v_2^*)\in C\big([0,T_0],\Chi_{2\sigma}^0\big)$, for some $T_0>0$ and a  $\sigma>0$,
be a solution of the multiphase WMEs (\ref{wmes-Mform}), with
\[
\sup_{T\in[0,T_0]}\big\|(r_1^*,v_1^*,r_2^*,v_2^*)(\cdot,T)\big\|_{\Chi_{2\sigma}^0}\leq C_{wh}\,,
\]
for some positive constant $C_{wh}$. Then for all $ C_1 > 0 $ there exist
positive constants $C_2$, $T_1$, $ \eps_0 $, and
solutions $\big(\check{r}_1,\check{v}_1, \check{r}_2,\check{v}_2 \big)$ of
(\ref{check-ru}),
such that
\[
\sup_{T\in[0,T_1]}\big\|
\big(\check{r}_1,\check{v}_1,\check{r}_2,\check{v}_2 \big)(X,T,\eps)-
\big(r_1^*,v_1^*,r_2^*,v_2^*\big)(X,T)\big\|_{\Chi_{\sigma}^0} \leq C_2\eps^2
\quad\mbox{for all}\ \eps\in(0,\eps_0)\,.
\]
\end{theorem}

For the subsequent discussion of the phase in the next section we need the following formulation of our approximation result.
\begin{theorem}
	\label{TheoremD}
	Fix $ n \in \N_0 $ and
let $ \widetilde \bu^n(T) = (r_1^{*,n},v_1^{*,n},r_2^{*,n},v_2^{*,n}\big)(T) $ be the higher order approximation constructed
in \S\ref{sec-approx}.
 Then there
	exist
	positive constants $C_2$, $T_1$, $ \eps_0 $, and
	 solutions $ \big(\check{r}_1,\check{v}_1,\check{r}_2,\check{v}_2 \big) $ of (\ref{check-ru}) with
	\[
\sup_{T\in[0,T_1]}\big\|
\big(\check{r}_1,\check{v}_1,\check{r}_2,\check{v}_2 \big)(X,T,\eps)-
\big(r_1^{*,n},v_1^{*,n},r_2^{*,n},v_2^{*,n}\big)(X,T,\eps)\big\|_{\Chi_{\sigma}^0} \leq C_2\eps^{2n+2}
\quad\mbox{for all}\ \eps\in(0,\eps_0)\,.
\]
\end{theorem}

\section{Concluding Remarks}
\setcounter{equation}{0}
\label{sec-cr}

Validity of multiphase WMEs deduced from CNLS equations has been proved. As far as we are aware this is the first proof of validity for multiphase
WMEs. The theory shows that multiphase modulation, at least in the case
where the underlying equation has a toral symmetry, is robust.
Only two-phase solutions have been considered, but extension to any finite number of phases, at least in the context of $N$-coupled NLS equations
reduced to $N$-phase WMEs, is conceivable.

The proof is independent of the phases, but the role of phases can be seen by
going back to the original $ \Psi_1 $, $ \Psi_2 $ variables in the
CNLS equations \eqref{cnls}.
Integration is necessary for the
reconstruction of the phases $\phi_j^{\beta}$ from  $v_j^{\beta} = \partial_X \phi_j^{\beta}$, and so only
a local in space approximation result is possible.
In particular, we have the following corollary of Theorem \ref{TheoremD}.

\begin{corollary} \label{cor1}
For all $b \in [0,2n+1]$ we have
\begin{eqnarray*}
\lefteqn{\sup\limits_{t\in[0,T_1/{\varepsilon}]}\sup\limits_{|x| \leq \varepsilon^{-b}}
    \Big| (\Psi_1,\Psi_2)(x,t) \exp(- i \phi(0,\eps t))} \\
&& - \left(\exp \left(r_1^{*,n}(\varepsilon x,\varepsilon t) + i  \int^{x}_0
    v_j^{*,n}(\varepsilon x',\varepsilon t)d x'+ i \omega t \right)\right)_{j=1,2}\Big|
    \leq C \varepsilon^{2n+1-b},
\end{eqnarray*}
where $\phi$ satisfies $\sup\limits_{T\in[0,{T_1}]}
|\phi(0,T) | = \mathcal{O}(1/\varepsilon)$.
\end{corollary}
\noindent Further detail on phase estimates in the
case of one-phase wavetrains can be found in \S2 of \cite{ds09}.

The Whitham (geometric optics) approximation is also used in
dissipative systems with underlying conservation laws,
and a validity proof has been given for
the case of one-phase wavetrains (e.g.\ \textsc{Johnson et al.}~\cite{JNRZ15}
and references therein).
However, an approximation theorem in the sense of Theorem \ref{TheoremC}
for dissipative systems is still work in progress \cite{HdRS20}.

Finally, the Cauchy-Kowalevskaya theorem and spaces of analytic
functions are essential when the WMEs are mixed or elliptic.  However,
a question arises in the case when the WMEs are hyperbolic.
Can the use of analytic functions,
and the additional restriction on the approximation time $ T_1 < T_0 $,
be avoided in the hyperbolic case, where the WMEs are locally well
posed?
In particular, can the validity result be proved in Sobolev spaces?
These questions, in the case of multiphase hyperbolic WMEs, remain open.

\appendix
\section{Auxiliary Estimates in Gevrey Spaces}
\label{app-a}

{\color{black}
  In this appendix we record the auxiliary estimates in Gevrey spaces
  which are used
  in the paper. Although all of these estimates are more or less standard,
  there is not standard reference, and so we sketch their proof here for
  the convenience of the reader. We start with the estimate for the point-wise product.
  \begin{lemma}\label{LemA.uv} Let $u,v\in G_\sigma^s(\R)$ for some $s,\sigma\ge0$ and
   $\kappa>1/2$. Then the
   following estimate holds:
\begin{equation}\label{A.mult}
\|uv\|_{G_\sigma^s}\le C_{s,\kappa}\(\|u\|_{G^s_\sigma}\|v\|_{G^\kappa_\sigma}+
\|u\|_{G^\kappa_\sigma}\|v\|_{G^s_\sigma}\)\,,
\end{equation}
for some positive constant $C_{s,\sigma}$.
  \end{lemma}
\begin{proof} We use Young's convolution inequality in the form
$$
\|f*g\|_{L^2}\le\|f\|_{L^2}\|g\|_{L^1}\,,
$$
together with the following inequalities,
which follow in turn from the triangle inequality,
$$
(1+|\xi|^{2s})^{1/2}\le C_s\((1+|\eta|^{2s})^{1/2}+(1+|\xi-\eta|^{2s})^{1/2}\),\ \ e^{\sigma(1+|\xi|)}\le e^{\sigma(1+|\eta|)}e^{\sigma(1+|\eta-\xi|)}
  $$
to get
\begin{multline}\label{A.big}
|e^{\sigma(|\xi|+1)}(1+|\xi|^{2s})^{1/2}(\hat u*\hat v)(\xi)|\le\\\le
C_s\int_\R \(e^{\sigma(|\eta|+1)}(1+|\eta|^{2s})^{1/2}|\hat u(\eta)|\)
\(e^{\sigma(1+|\xi-\eta|)}|\hat v(\xi-\eta)|\)\,d\eta+\\+
C_s\int_\R \(e^{\sigma(|\xi-\eta|+1)}(1+|\xi-\eta|^{2s})^{1/2}|\hat v(\xi-\eta)|\)
\(e^{\sigma(1+|\eta|)}|\hat u(\eta)|\)\,d\eta
\end{multline}
and
\begin{multline}
\|uv\|_{G^s_\sigma}=\|e^{\sigma(1+|\xi|)}(1+|\xi|^{2s})^{1/2}(\hat u*\hat v)(\xi)\|_{L^2}\le\\\le
C_s\|u\|_{G^s_\sigma}\|e^{\sigma(1+|\xi|)}\hat v(\xi)\|_{L^1}+C_s\|v\|_{G^s_\sigma}
\|e^{\sigma(1+|\xi|)}\hat u(\xi)\|_{L^1}.
\end{multline}
Finally, due to Cauchy-Schwarz inequality and the fact that $\kappa>\frac12$, we arrive at
\begin{multline}
\int_\R e^{\sigma(1+|\xi|)}|\hat u(\xi)|\,d\xi=
\int_\R\(e^{\sigma(1+|\xi|)}(1+|\xi|^{2\kappa})^{1/2}|\hat u(\xi)|\)
\(1+|\xi|^{2\kappa}\)^{-1/2}\,d\kappa\le\\\le
\|u\|_{G_\sigma^\kappa}\|(1+|\xi|^{2\kappa})^{-1}\|_{L^1}\le C_\kappa \|u\|_{G^\kappa_\sigma}.
\end{multline}
This estimate, together with \eqref{A.big} and with the analogous estimate for $v$ finish the proof of the lemma.
\end{proof}
\begin{corollary} Let $u\in G^s_\sigma(\R)$ for some $\sigma>\frac12$ and let $\varphi:\Bbb C\to\Bbb C$ be
an entire function such that $\varphi(0)=0$. Then, $\varphi(u)\in G^s_\sigma$ and there exists an entire
function $\varphi_s(z)$, $z\in\Bbb C$ such that $\varphi_s(0)=0$, $\varphi_s(x)$ is positive and monotone increasing
 for $x\in\R_+$ such that
$$
\|\varphi(u)\|_{G^s_\sigma}\le \varphi_s(\|u\|_{G^s_\sigma}).
$$
\end{corollary}
\begin{proof} Since $\varphi$ is entire, it
  can be expressed in terms of the convergent Taylor expansion
$$
\varphi(z)=\sum_{n=1}^\infty a_nz^n,\  z\in\Bbb C
$$
($a_0=0$ since $\varphi(0)=0$). Thus, using \eqref{A.mult} with $\kappa=s$, we get
$$
\|\varphi(u)\|_{G^s_\sigma}\le\sum_{n=1}^\infty |a_n|\|u^n\|_{G^s_n}\le
\sum_{n=1}^\infty |a_n| (2C_{s,s})^n\|u\|_{G^s_n}^n:=\varphi_s(\|u\|_{G^s_\kappa})\,.
$$
Since $\varphi$ is entire, we have $\lim_{n\to\infty}\sqrt[n]{|a_n|}=0$ and therefore $\varphi_s(z)$ is also
an entire function. Obviously $\varphi_s(0)=0$ and $\varphi_s'(x)\ge0$ if $x\ge0$. This finishes
 the proof of the corollary.
\end{proof}

To conclude, we state one more useful estimate which compares Gevrey norms with different exponents.

\begin{lemma}\label{LemA.pd} Let $u\in G^s_\sigma$ for some $\sigma>0$, $s\ge0$. Then
$$
\|u\|_{G^{s+p}_{\sigma-\delta}}\le C(p,\delta)\|u\|_{G^s_\sigma},
$$
where $C(p,\delta)$ is independent of $s$ and $\sigma$.
\end{lemma}
\begin{proof} Indeed, by definition
$$
\|u\|^2_{G^{s+p}_{\sigma-\delta}}=\int_\R e^{2(\sigma-\delta)(1+|\xi|)}
(1+|\xi|^{2(s+p)})|\hat u(\xi)|^2\,d\xi\le
\max_{\xi\in\R}\{e^{-2\delta(1+|\xi|)}(1+|\xi|^{2p})\}\|u\|^2_{G^s_\sigma}\,.
$$
Thus, the inequality holds with
$$
C(p,\delta):=\max_{\xi\in\R}\{e^{-\delta(1+|\xi|)}(1+|\xi|^{2p})^{1/2}\}\,.
$$
\end{proof}
}


\end{document}